\documentclass[10pt,a4paper,reqno]{amsart}
\usepackage{graphicx}
\usepackage[nocompress]{cite}
\usepackage{tikz-cd}
\usepackage{subcaption}
\usepackage{mathrsfs}
\usepackage{amsopn,amssymb,amsmath}
\usepackage{url}
\usepackage{subcaption}
\usepackage{verbatim}
\usepackage[hidelinks]{hyperref}
\theoremstyle{definition}
\newtheorem{theorem}{Theorem}[section]
\newtheorem{prop}[theorem]{Proposition}
\newtheorem{defn}[theorem]{Definition}
\newtheorem{cor}[theorem]{Corollary}
\newtheorem{eg}[theorem]{Example}
\newtheorem{lemma}[theorem]{Lemma}
\newtheorem{remark}[theorem]{Remark}
\numberwithin{equation}{section}

\DeclareMathOperator{\Ima}{Im}

\DeclareMathOperator{\red}{red}

\newcommand{\Z}{\mathbb{Z}}
\newcommand{\R}{\mathbb{R}}

\newcommand{\F}{\mathbb{F}}

\begin{document}

\title{Computational Tools in Weighted Persistent Homology}

\author{Shiquan Ren}
\address{Yau Mathematical Sciences Center, Tsinghua University, Beijing 100084, China}
\email{sren@u.nus.edu}
\thanks{The project was supported in part by the Singapore Ministry of Education research grant (AcRF Tier 1 WBS No.~R-146-000-222-112). The first author was supported in part by the National Research Foundation, Prime Minister's Office, Singapore under its Campus for Research Excellence and Technological Enterprise (CREATE) programme. The second author was supported in part by the President's Graduate Fellowship of National University of Singapore. The third author was supported by a grant (No.~11329101) of NSFC of China.}

\author{Chengyuan Wu$^*$}
\address{Department of Mathematics, National University of Singapore, Singapore 119076}
\email{wuchengyuan@u.nus.edu}
\thanks{$^*$Corresponding author}

\author{Jie Wu}
\address{School of Mathematics and Information Science, Hebei Normal University, Hebei 050024, China; and Department of Mathematics, National University of Singapore, Singapore 119076}
\email{matwuj@nus.edu.sg}





\begin{abstract}
In this paper, we study further properties and applications of weighted homology and persistent homology. We introduce the Mayer-Vietoris sequence and generalized Bockstein spectral sequence for weighted homology. For applications, we show an algorithm to construct a filtration of weighted simplicial complexes from a weighted network. We also prove a theorem that allows us to calculate the mod $p^2$ weighted persistent homology given some information on the mod $p$ weighted persistent homology.
\end{abstract}

\maketitle
\section{Introduction}
Persistent homology is a recent branch of applied algebraic topology that has applications in data analysis \cite{bubenik2015statistical}, image processing and recognition \cite{Carlsson2008,adcock2016ring}, and more \cite{bendich2016persistent,dewoskin2010applications}. It is also a subject of active research, from both the computational \cite{dlotko2014simplification} and theoretical \cite{bubenik2007statistical,bubenik2014categorification} points of view.

The purpose of this paper is to reformulate some classical computational tools in homology theory in the context of weighted persistent homology. An important computational tool in homology theory is the Mayer-Vietoris sequence which can largely shorten the computations by handling well the subcomplexes. The Bockstein spectral sequence is a classical tool for recovering the integral homology from mod $p$ homology. In the situation of weighted persistent homology, it is convenient to consider the chains with coefficients in $\Z/2\Z$ (see \cite{Zomorodian2005,edelsbrunner2012persistent,boissonnat2014computing}). It is important to explore the Bockstein spectral sequence on weighted persistent homology as a tool to recover the integral weighted persistent homology so that one can obtain more topological information on weighted data.

In this paper, we study further properties and applications of weighted homology and persistent homology. Weighted simplicial homology \cite{Dawson1990,ren2018weighted} is a generalization of simplicial homology, that reduces to the usual simplicial homology when all the simplices have the same nonzero weight. For weighted simplicial complexes, we allow weights in a commutative ring $R$ with unity. When considering weighted homology, we require the ring $R$ to be an integral domain. In \cite{ren2018weighted}, it is shown that weighted persistent homology can tell apart filtrations that ordinary persistent homology does not distinguish. For example, if there is a point considered as special, weighted persistent homology can tell when a cycle containing the point is formed or has disappeared. Hence, weighted persistent homology is a richer invariant than persistent homology.

Our approach to weighted persistent homology is to weight the boundary map. There are also various other approaches to adding weight to persistent homology \cite{bell2017weighted,Petri2013,buchet2016efficient,edelsbrunner2012persistent}.

The Mayer-Vietoris sequence is an important tool in algebraic topology to study the homology of a space. In Section \ref{sec:mayer}, we state and verify the Mayer-Vietoris sequence for weighted homology.

In Section \ref{sec:filt}, we show that a descending chain of ideals gives rise to a filtration of weighted simplicial complexes. An application is an algorithm (Subsection \ref{subsec:filt}) to construct a filtration of weighted simplicial complexes from a weighted network (weighted graph). This is related to the concept of Weight Rank Clique filtration \cite{Petri2013} which is used in the study of complex networks \cite{boccaletti2006complex,strogatz2001exploring,albert2002statistical,Petri2013}. A key feature is that in the process, we only construct the clique complex once, as opposed to the Weight Rank Clique filtration where multiple constructions of the clique complex is needed. In Subsection \ref{subsec:stanley}, we also illustrate an application in relation to Stanley-Reisner theory where a filtration could be set up so as to gather new information about the weighted simplicial complexes.

Next, we prove that over a field $\mathbb{F}$, the weighted homology groups $H_n(K,w;\mathbb{F})$ are isomorphic to the usual unweighted homology groups $H_n(K;\mathbb{F})$. This is described in greater detail in Section \ref{sec:field}.

In Section \ref{sec:bockstein}, we develop the Bockstein spectral sequence for weighted homology. The motivation behind using the Bockstein spectral sequence is that in persistent homology algorithms \cite{Zomorodian2005,boissonnat2014computing}, the homology is usually computed with field coefficients. However, the integral homology groups contain more information than the homology groups with field coefficients. The Bockstein spectral sequence allows us to ``unravel'' the integral homology from the mod $p$ homology. In the process, we prove a theorem (Theorem \ref{cor:ptop2}) that allow us to calculate the mod $p^2$ weighted persistent homology provided some conditions on the mod $p$ persistent homology are satisfied.

The final part of this paper (Section \ref{sec:genbockstein}) is about the generalized Bockstein spectral sequence, where we consider coefficients in an integral domain $R$. A potential application is in algebraic geometry where recently there has been some interest in the usage of weighted simplicial complexes \cite{doran2016simplicial,gonzalez2017balanced} with weights in a ring $R$.

\section{Background}
In this section, we review the background necessary for the subsequent sections. We begin by reviewing weighted simplicial homology \cite{Dawson1990,ren2018weighted}, and then weighted persistent homology \cite{ren2018weighted}. 
\subsection{Weighted Simplicial Homology}
Weighted simplicial homology \cite{Dawson1990,ren2018weighted} is a generalization of simplicial homology. Every simplex has a weight in a ring $R$, and the boundary map is weighted accordingly. When all the simplices have the same weight $a\in R\setminus\{0\}$, the resulting weighted homology is the same as the usual simplicial homology. We list some of the key definitions and results below.
\begin{defn}[{\cite[p.~2666]{ren2018weighted}}]
\label{divcond}
A weighted simplicial complex (or WSC for short) is a pair $(K,w)$ consisting of a simplicial complex $K$ and a weight function $w: K\to R$, where $R$ is a commutative ring, such that for any $\sigma_1,\sigma_2\in K$ with $\sigma_1\subseteq\sigma_2$, we have $w(\sigma_1)\mid w(\sigma_2)$.
\end{defn}
\begin{theorem}[{\cite[p.~2668]{ren2018weighted}}]
\label{idealcomplex}
Let $I$ be an ideal of a commutative ring $R$. Let $(K,w)$ be a weighted simplicial complex, where $w:K\to R$ is a weight function. Then $K\setminus w^{-1}(I)$ is a simplicial subcomplex of $K$. 
\qed
\end{theorem}

For the definition of homology of weighted simplicial complexes \cite[p.~2672]{ren2018weighted}, we require $R$ to be an integral domain with 1. 

\begin{defn}[{\cite[p.~2674]{ren2018weighted}}]
\label{boundary}
The \emph{weighted boundary map} $\partial_n: C_n(K)\to C_{n-1}(K)$ is the map: \[\partial_n(\sigma)=\sum_{i=0}^n\frac{w(\sigma)}{w(d_i(\sigma))}(-1)^id_i(\sigma)\] where the \emph{face maps} $d_i$ are defined as: \[d_i(\sigma)=[v_0,\dots,\widehat{v_i},\dots,v_n]\qquad\text{(deleting the vertex $v_i$)}\]  for any $n$-simplex $\sigma=[v_0,\dots,v_n]$. 
\end{defn}
\begin{theorem}[{\cite[p.~2676]{ren2018weighted}}]
\label{fdcommute}
Let $f:K\to L$ be a simplicial map. Then $f_\sharp\partial=\partial f_\sharp$, where $\partial$ refers to the relevant weighted boundary map.
\qed
\end{theorem}
\begin{defn}[{\cite[p.~2677]{ren2018weighted}}]
We define the \emph{weighted homology} of a WSC to be
\begin{equation*}
H_n(K,w):=\ker(\partial_n)/\Ima(\partial_{n+1}),
\end{equation*}
where $\partial_n$ is the weighted boundary map.
\end{defn}
\begin{prop}[{\cite[p.~2679]{ren2018weighted}}]
If all the simplices in $(K,w)$ have the same weight $a\in R\setminus\{0\}$, the weighted homology functor is the same as the usual simplicial homology functor.
\qed
\end{prop}
\subsection{Weighted Persistent Homology}
Given a weighted filtered complex $(\mathcal{K},w)=\{(K^i, w)\}_{i \geq 0}$, for the $i$th complex $K^i$ we have the associated weighted boundary maps $\partial_k^i$ and groups $C_k^i$, $Z_k^i$, $B_k^i$, $H_k^i$ for all integers $i, k\geq 0$, as shown in \cite{ren2018weighted}.

\begin{defn}[{\cite[p.~2679]{ren2018weighted}}]
The weighted boundary map $\partial_k^i$, where $i$ denotes the filtration index, is the weighted boundary map of the $i$th complex $K^i$. That is, $\partial_k^i$ is the map $\partial_k^i: C_k(K^i,w)\to C_{k-1}(K^i,w)$. The \emph{chain group} $C_k^i$ is the group $C_k(K^i,w)$. The \emph{cycle group} $Z_k^i$ is the group $\ker (\partial_k^i)$, while the \emph{boundary group} $B_k^i$ is the group $\Ima(\partial_{k+1}^i)$. The \emph{homology group} $H_k^i$ is the quotient group $Z_k^i/B_k^i$.
\end{defn}

\begin{defn}[{\cite[p.~2680]{ren2018weighted}}]
The \emph{$p$-persistent $k$th homology group of $(\mathcal{K},w)=\{(K^i, w)\}_{i \geq 0}$} is defined as
\begin{equation*}
H_k^{i,p}(\mathcal{K},w):=Z_k^i/(B_k^{i+p}\cap Z_k^i).
\end{equation*}
\end{defn}

\section{The Mayer-Vietoris Sequence and Weighted Homology}
\label{sec:mayer}
The Mayer-Vietoris sequence for weighted simplicial homology was first studied briefly in \cite[p.~235]{Dawson1990}. We prove that the Mayer-Vietoris sequence is exact for weighted simplicial homology, using an approach based on \cite[p.~142]{Munkres1984}, which is different from the approach given in \cite{Dawson1990}. In this section, we let $R$ be an integral domain.

W will need the following Lemma \ref{zigzag}, which is also known as the \emph{Zig-zag Lemma}\cite[p.~136]{Munkres1984}.
\begin{lemma}[Zig-zag Lemma]
\label{zigzag}
Let $\mathscr{C}=\{C_p,\partial_C\}$, $\mathscr{D}=\{D_p,\partial_D\}$ and $\mathscr{E}=\{E_p,\partial_E\}$ be chain complexes, and let $\phi$, $\psi$ be chain maps such that \[0\to\mathscr{C}\xrightarrow{\phi}\mathscr{D}\xrightarrow{\psi}\mathscr{E}\to 0\] is a short exact sequence of chain complexes.

Then there is a long exact homology sequence
\begin{equation*}
\dots\to H_p(\mathscr{C})\xrightarrow{\phi_*}H_p(\mathscr{D})\xrightarrow{\psi_*}H_p(\mathscr{E})\xrightarrow{\partial_*}H_{p-1}(\mathscr{C})\xrightarrow{\phi_*}H_{p-1}(\mathscr{D})\to\dots
\end{equation*}
where $\partial_*$ is induced by the boundary operator in $\mathscr{D}$.
\end{lemma}
\begin{proof}
A detailed proof can be found in \cite[p.~137]{Munkres1984}.
\end{proof}

The following is a generalization of the Mayer-Vietoris sequence in simplicial homology \cite[p.~142]{Munkres1984}.
\begin{theorem}[{cf.\ \cite[p.~142]{Munkres1984}}]
Let $(K,w)$ be a weighted simplicial complex, with weight function $w:K\to R$. Let $(K_0,w)$, $(K_1,w)$ be weighted subcomplexes such that $K=K_0\cup K_1$. Let $A=K_0\cap K_1$. Then there is an exact sequence
\begin{equation*}
\dots\to H_p(A,w)\to H_p(K_0,w)\oplus H_p(K_1,w)\to H_p(K,w)\to H_{p-1}(A,w)\to\dots
\end{equation*}
which we call the \emph{Mayer-Vietoris sequence for weighted homology}.
\end{theorem}
\begin{proof}
The proof is similar to \cite[p.~142]{Munkres1984}.
\end{proof}
\section{Further Properties of Weighted Simplicial Complexes}
\label{sec:filt}
In this section, let $I$ be an ideal of a commutative ring $R$, and let $(K,w)$ be a weighted simplicial complex, with weight function $w: K\to R$. In Theorem \ref{idealcomplex}, it is proved that $K\setminus w^{-1}(I)$ is a simplicial subcomplex of $K$. We explore this idea further in this section. We show that a descending chain of ideals gives rise to a filtration, and relate this filtration to the concept of Weight Rank Clique filtration \cite{Petri2013} which is used in the study of complex networks \cite{boccaletti2006complex,strogatz2001exploring,albert2002statistical,Petri2013}.
\begin{theorem}
\label{idealfiltration}
Let $R=I_0\supseteq I_1\supseteq \dots\supseteq 0$ be a descending chain of ideals of $R$. Let $(K,w)$ be a WSC (with weight function $w:K\to R$) such that all simplices have nonzero weights. Define
\begin{equation*}
L^i=K\setminus w^{-1}(I_i)
\end{equation*}
for $i\in\mathbb{Z}_{\geq 0}$.

Then 
\begin{equation*}
\emptyset=L^0\subseteq L^1\subseteq\dots\subseteq K\setminus w^{-1}(0)=K
\end{equation*}
is a filtration of $K$.
\end{theorem}
\begin{proof}
By Theorem \ref{idealcomplex}, all of the $L^i$ are simplicial complexes of $K$. Since $w^{-1}(R)=K$, thus $L^0=K\setminus K=\emptyset$. Also since $w^{-1}(0)=\emptyset$, thus $K\setminus w^{-1}(0)=K$.

Let $\sigma\in L^i$ for some $i\in\mathbb{Z}_{\geq 0}$. Then $w(\sigma)\notin I_i\supseteq I_{i+1}$. Hence $w(\sigma)\notin I_{i+1}$, and so $\sigma\in L^{i+1}$. Hence $L^i\subseteq L^{i+1}$ for all $i\in\mathbb{Z}_{\geq 0}$.
\end{proof}

We now discuss the filtration in Theorem \ref{idealfiltration} in relation to Weighted Rank Clique filtration \cite[p.~7]{Petri2013}.
\begin{defn}[{\cite[p.~7]{Petri2013}}]
The \emph{Weight Rank Clique filtration} on a weighted network (weighted graph) $\Omega$ combines the clique complex (simplicial complex built from the cliques of a graph) construction with a thresholding on weights following three main steps.
\begin{itemize}
\item Rank the weight of links (edges of the graph $\Omega$) from $w_\text{max}:=\epsilon_1$ to $w_\text{min}$. The discrete decreasing parameter $\epsilon_t$, \[w_\text{max}=\epsilon_1\geq \epsilon_2 \geq \dots\geq w_\text{min},\] indexes the sequence. 
\item At each step $t$ of the decreasing edge ranking we consider the thresholded graph $G(\epsilon_t)$, i.e.\ the subgraph of $\Omega$ with links of weights larger than $\epsilon_t$.
\item For each graph $G(\epsilon_t)$ we build the clique complex $K(G(\epsilon_t))$.
\end{itemize}
The clique complexes are nested along the growth of $t$ and determine the weight rank clique filtration. A key feature is that links with larger weights will appear earlier in $G(\epsilon_t)$, while links with smaller weights will appear later in $G(\epsilon_t)$.
\end{defn}
We now consider an analogous definition of Weighted Rank Clique filtration for the case of weighted simplicial complexes, which we call \emph{Weight Rank Simplicial filtration}. We drop the clique complex construction part of Weight Rank Clique filtration since we already start with a simplicial complex as our initial object. Also, since the weights of links are indexed by a discrete parameter $\epsilon_t$, in our analogous definition we may choose our weights of simplices to lie in a subset of a discrete ordered ring, say $\mathbb{Z}_{>0}\subseteq R=\mathbb{Z}$.
\begin{defn}[Weight Rank Simplicial filtration]
\label{wrsf}
The \emph{Weight Rank Simplicial filtration} on a weighted simplicial complex $(K,w)$ with $w: K\to\mathbb{Z}_{>0}\subseteq\mathbb{Z}$ is defined via a thresholding on weights following two main steps.
\begin{itemize}
\item Rank the weights of simplices of $K$ from $w_\text{min}:=\epsilon_1$ to $w_\text{max}$. The discrete increasing paramenter $\epsilon_t\in\mathbb{Z}_{>0}$, \[w_{min}=\epsilon_1\leq \epsilon_2\leq\dots\leq w_\text{max},\] indexes the sequence.
\item At each step $t$ of the increasing weight ranking we consider the thresholded simplicial complex $L(\epsilon_t)$, i.e.\ the subcomplex\footnote{Let $\sigma\in L(\epsilon_t)$. For any nonempty $\tau\subseteq\sigma$, we have $w(\tau)\mid w(\sigma)$ and hence $w(\tau)\leq w(\sigma)<\epsilon_t$. Thus $\tau\in L(\epsilon_t)$, and hence $L(\epsilon_t)$ is indeed a subcomplex.} of $(K,w)$ consisting of all simplices of weights smaller than $\epsilon_t$.
\end{itemize}
The Weight Rank Simplicial filtration can be viewed as a sublevel set filtration. The subcomplexes $\emptyset=L(\epsilon_1)\subseteq L(\epsilon_2)\subseteq\dots\subseteq L(w_\text{max})\subseteq K$ clearly form a filtration. A key feature is that simplices with smaller weights will appear earlier in $L(\epsilon_t)$. 
\end{defn}
\begin{remark}
Note that the weights of simplices in $(K,w)$ have to satisfy the divisibility condition in Definition \ref{divcond}.
\end{remark}
In Weight Rank Simplicial filtration, the simplices with smaller weights appear earlier in the filtration which is the direct opposite case to the Weight Rank Clique filtration. In practical applications, this is unlikely to be an issue as we can always reverse the ordering of the weights if necessary, by assigning smaller weights (instead of larger weights) to the simplices that we want to appear first. We now prove a theorem applying Theorem \ref{idealfiltration} to relate a descending chain of ideals in $\mathbb{Z}$ with Weight Rank Simplicial filtration.
\begin{theorem}
\label{idealequiv}
Let $(K,w)$ be a WSC with $w: K\to\mathbb{Z}_{>0}$, such that the weights of simplices of $K$ are totally ordered by division. Consider the Weight Rank Simplicial filtration as described in Definition \ref{wrsf}.
\begin{itemize}
\item Rank the weights of simplices of $K$ from $w_\text{min}:=w_1$ to $w_\text{max}$, \[w_\text{min}=w_1\mid w_2\mid\dots\mid w_\text{max},\] indexed by the parameter $w_t\in\mathbb{Z}_{>0}$.
\item At each step $t$ of increasing weight ranking we consider the thresholded simplicial complex $L(w_t)$, i.e.\ the subcomplex of $(K,w)$ with simplices of weights smaller than $w_t$.
\end{itemize}
Then the filtration $\mathscr{F}_1$,
\begin{equation*}
\emptyset=L(w_1)\subseteq L(w_2)\subseteq\dots\subseteq L(w_\text{max})\subseteq K,
\end{equation*}
is the same as the filtration $\mathscr{F}_2$,
\begin{equation*}
\emptyset=L^1\subseteq L^2\subseteq\dots\subseteq K,
\end{equation*}
where $L^t=K\setminus w^{-1}(w_t\mathbb{Z})$ for $t\geq 1$, $t\in\mathbb{Z}$. 
\end{theorem}
\begin{proof}
Note that \[\mathbb{Z}\supseteq w_1\mathbb{Z}\supseteq w_2\mathbb{Z}\supseteq\dots\supseteq 0\] is a descending chain of ideals of $\mathbb{Z}$. It is given that all simplices of $(K,w)$ have nonzero weights. By Theorem \ref{idealfiltration}, $\mathscr{F}_2$ is a filtration of $K$. Note that \[L^1=K\setminus w^{-1}(w_1\mathbb{Z})=K\setminus K=\emptyset=L(w_1).\]
Now it suffices to prove that $L(w_t)=L^t$, for all $t\geq 1$, $t\in\mathbb{Z}$.

We have the following equivalent statements.
\begin{align*}
\sigma\in L(w_t)&\iff w(\sigma)<w_t\\
&\iff w_t\nmid w(\sigma)\\
&\iff w(\sigma)\notin w_t\mathbb{Z}\\
&\iff \sigma\in K\setminus w^{-1}(w_t\mathbb{Z})\\
&\iff \sigma\in L_t.
\end{align*}
Hence we have shown that $L(w_t)=L^t$ as desired.
\end{proof}
\subsection{Application}
\label{subsec:filt}
The Weighted Rank Simplicial filtration in Definition \ref{wrsf} provides an alternative way to construct a filtration of (weighted) simplicial complexes from a weighted network (weighted graph) $\Omega$. First we construct the clique complex $K$ from $\Omega$, and assign postive integer weights to make $K$ into a weighted simplicial complex $(K,w)$. This can be done as follows:
\begin{itemize}
\item Set all 0-simplices (vertices) in $K$ to have weight 1.
\item Rank the weight of links (edges) of $\Omega$ in increasing/decreasing order (depending on which edges the user wishes to appear first in the resulting filtration).
\item Set the weight of each 1-simplex (edge) in $K$ to be $2^k$, where $k$ is its rank in the weight ranking of $\Omega$ (edges can have the same rank if they have the same weight in $\Omega$).
\item For higher dimensional simplices, its weight is set to be the product of all the weights of the 1-simplices contained in it.
\end{itemize}

Then we carry out Weight Rank Simplicial filtration to obtain a filtration $\mathscr{F}$ of $(K,w)$. The filtration $\mathscr{F}$ can be described in terms of complements of preimage of ideals, as shown in Theorem \ref{idealequiv}.

Though the weights of the form $2^k$ can be very large integers, in practice we only need to store and compute the exponent $k$. Divisibility can be checked easily since $2^{k_1}\mid 2^{k_2}$ if and only if $k_1\leq k_2$.

Note that in the entire process, we only construct the clique complex once. In general, it is desirable to reduce the number of times we construct the clique complex \cite{zomorodian2010tidy}.

\subsection{Filtration related to the Stanley-Reisner Ideal of a Weighted Simplicial Complex}
\label{subsec:stanley}
We show a type of filtration that has potential relations to Stanley-Reisner theory\cite{stanley2007combinatorics,miller2004combinatorial} which is an important topic in algebraic combinatorics, combinatorial commutative algebra and algebraic geometry. 

\begin{defn}[cf.\ {\cite[p.~53]{stanley2007combinatorics}\cite[p.~5]{miller2004combinatorial}}]
Let $\Delta$ be a finite simplicial complex with vertex set $V=\{x_1,\dots,x_n\}$. Let $k$ be a field. The Stanley-Reisner ring $k[\Delta]$ is defined by
\begin{equation*}
k[\Delta]=k[x_1,\dots,x_n]/I_\Delta,
\end{equation*}
where the Stanley-Reisner ideal is given by
\begin{equation*}
I_\Delta=(x_{i_1}x_{i_2}\dots x_{i_r}\mid \{x_{i_1}\},\dots,\{x_{i_r}\}\in\Delta, \{x_{i_1},x_{i_2},\dots,x_{i_r}\}\notin\Delta).
\end{equation*}

In other words, $k[\Delta]$ is obtained from the polynomial ring $k[x_1,\dots,x_n]$ by quotienting out the ideal $I_\Delta$ generated by the square-free monomials corresponding to the non-faces of $\Delta$.
\end{defn}
\begin{remark}
We have added the condition $\{x_{i_1}\},\dots,\{x_{i_r}\}\in\Delta$ in the definition of the Stanley-Reisner ideal $I_\Delta$ to follow the usual convention that every vertex $\{x_i\}$ should be a simplex in $\Delta$. Hence, none of the variables $\{x_i\}$ belong to $I_\Delta$.
\end{remark}

\begin{defn}
Let $k$ be a field. Let $(\Delta_1,w)$ be a WSC where $\Delta_1$ is a finite simplicial complex with vertex set $V=\{x_1,\dots,x_n\}$, and $w:K\to k[x_1,\dots,x_n]$ is a weight function.

By Theorem \ref{idealcomplex}, $\Delta_2:=\Delta_1\setminus w^{-1}(I_{\Delta_1})$ is a subcomplex of $\Delta_1$. By induction, define $\Delta_{i+1}:=\Delta_i\setminus w^{-1}(I_{\Delta_i})$, which is a subcomplex of $\Delta_i$. Then we have a filtration
\begin{equation*}
\mathcal{F}_{(\Delta_1,w)}:=\{\dots\subseteq\Delta_{i+1}\subseteq\Delta_i\subseteq\Delta_{i-1}\subseteq\dots\subseteq\Delta_1\}
\end{equation*}
which we call the \emph{Stanley-Reisner filtration} of the WSC $(\Delta_1,w)$.
\end{defn}

\begin{remark}
Since $\Delta_1$ is a \emph{finite} simplicial complex, the Stanley-Reisner filtration must eventually stabilize. That is, that there must exist an integer $k$ such that $\Delta_i=\Delta_k$ for all $i\geq k$. For convenience, we may truncate the Stanley-Reisner filtration by letting the first term of the filtration be $\Delta_k$. Hence, in practice we may always assume that the Stanley-Reisner filtration has finitely many terms.
\end{remark}

We prove a few propositions which show that the Stanley-Reisner filtration $\mathcal{F}_{(\Delta_1,w)}$ can give information about the weighted simplicial complex $(\Delta_1,w_1)$.

\begin{prop}
\label{prop:stanleytrivial}
Let $(\Delta_1,w)$ be a WSC where $\Delta_1$ is a simplex $\{x_1,\dots,x_n\}$, and $w(\sigma)\neq 0$ for all $\sigma\in\Delta_1$. Then $\mathcal{F}_{(\Delta_1,w)}=\{\Delta_1\subseteq\dots\subseteq\Delta_1\}$ is the trivial filtration consisting of only $\Delta_1$.
\end{prop}
\begin{proof}
If $\Delta_1$ is a simplex, then $I_{\Delta_1}=0$ since all faces $\{x_{i_1},x_{i_2},\dots,x_{i_r}\}$ lie in $\Delta_1$. Hence $w^{-1}(I_{\Delta_1})=\emptyset$ and $\Delta_2=\Delta_1$. By induction, $\Delta_i=\Delta_1$ for all $i$.
\end{proof}

The following contrapositive of Proposition \ref{prop:stanleytrivial} is useful for extracting some information about the WSC given the Stanley-Reisner filtration.
\begin{cor}
If $\mathcal{F}_{(\Delta_1,w)}$ is \emph{not} the trivial filtration $\{\Delta_1\subseteq\dots\subseteq\Delta_1\}$, then either $\Delta_1$ is \emph{not} a simplex $\{x_1,\dots,x_n\}$ or $w(\sigma)=0$ for some $\sigma\in\Delta_1$.
\qed
\end{cor}

Proposition \ref{prop:stanleytrivial} can also be stated in the language of weighted persistent homology.
\begin{prop}
Let $(\Delta_1,w)$ be a WSC where $\Delta_1$ is a simplex $\{x_1,\dots,x_n\}$, and $w(\sigma)\neq 0$ for all $\sigma\in\Delta_1$. Let $(\mathcal{K},w)=\{(K^i,w)\}_{i\geq 0}$ be the weighted filtered complex corresponding to $\mathcal{F}_{(\Delta_1,w)}$. (We renumber and relabel $\mathcal{F}_{(\Delta_1,w)}$ such that the index starts from 0 and call it $\{(K^i,w)\}_{i\geq 0}$.)

Then
\begin{equation}
H_k^{i,p}(\mathcal{K},w)=H_k(\Delta_1,w)
\end{equation}
for all $i$, $p$.
\end{prop}
\begin{proof}
By Proposition \ref{prop:stanleytrivial}, $\mathcal{F}_{(\Delta_1,w)}=(\mathcal{K},w)$ is the trivial filtration. Hence, the cycle group $Z_k^i$ is equal to $Z_k^0$ for all $i$, and the boundary group $B_k^{i+p}$ is equal to $B_k^0$ for all $i$, $p$. Thus $H_k^{i,p}(\mathcal{K},w)=H_k^{0,0}(\mathcal{K},w)=H_k(K^0,w)=H_k(\Delta_1,w)$ for all $i$, $p$.
\end{proof}

\begin{prop}
\label{prop:isolatedvertices}
Let $(\Delta_1,w)$ be a WSC where $\Delta_1$ consists of $n$ isolated vertices $\{x_1\},\dots,\{x_n\}$, where $n\geq 3$. Suppose $w(\{x_i\})=\frac{x_1x_2\dots x_n}{x_i}$ for all $i$.

Then $\mathcal{F}_{(\Delta_1,w)}=\{\emptyset\subseteq\dots\subseteq\emptyset=\Delta_2\subseteq\Delta_1\}$.
\end{prop}
\begin{proof}
Note that $I_{\Delta_1}=(x_ix_j\mid 1\leq i<j\leq n)$. Since $n\geq 3$, we observe that $w(\{x_i\})\in I_{\Delta_1}$ for all $i$. Thus $\Delta_2=\Delta_1\setminus w^{-1}(I_{\Delta_1})=\emptyset$. Subsequently, it is clear that $\Delta_i=\emptyset$ for $i\geq 2$.
\end{proof}

Similarly, Proposition \ref{prop:isolatedvertices} can be stated in the language of weighted persistent homology.

\begin{prop}
Let $(\mathcal{K},w)=\{(K^i,w)\}_{i\geq 0}$ be the weighted filtered complex corresponding to $\mathcal{F}_{(\Delta_1,w)}=\{\emptyset=\Delta_2\subseteq\Delta_1\}$, where $(\Delta_1,w)$ is the WSC satisfying the conditions in Proposition \ref{prop:isolatedvertices}.

Then $H_k^{0,1}(\mathcal{K},w)=0$ and
\begin{equation*}
H_k^{1,0}(\mathcal{K},w)=
\begin{cases}
\Z^n &\text{if $k=0$},\\
0 &\text{if $k\geq 1$}.
\end{cases}
\end{equation*}
\end{prop}
\begin{proof}
We have $H_k^{0,1}(\mathcal{K},w)=Z_k^0/(B_k^1\cap Z_k^0)=0$ since $K^0=\emptyset$ and thus $Z_k^0=0$.

We also have $H_k^{1,0}(\mathcal{K},w)=Z_k^1/(B_k^1\cap Z_k^1)=H_k(\Delta_1,w)$. Since $\Delta_1$ consists of $n$ isolated vertices, thus $H_k(\Delta_1,w)=\Z^n$ if  $k=0$ and $H_k(\Delta_1,w)=0$ if $k\geq 1$.
\end{proof}

Propositions \ref{prop:stanleytrivial} and \ref{prop:isolatedvertices} show that the Stanley-Reisner filtration can distingush between WSCs (with suitably chosen weights). In turn, weighted persistent homology is a possible tool to study the Stanley-Reisner filtration. In our brief discussion, we show that there is some promise in applying weighted persistent homology to study algebraic geometry / combinatorial commutative algebra through the connection with Stanley-Reisner theory.
\section{Weighted Homology over a field $\mathbb{F}$ with weight function $w:K\to\mathbb{F}$}
\label{sec:field}
In this section, let $K$ be a finite simplicial complex. Assume that both the coefficient ring $R$ and the codomain of the weight function $w:K\to R$ are the same field $R=\mathbb{F}$. We prove that the weighted homology groups $H_n(K,w;\mathbb{F})$ are isomorphic to the corresponding unweighted homology groups $H_n(K;\mathbb{F})$, for all $n$ and for all WSCs $(K,w:K\to R\setminus\{0\})$ where all weights of simplices are nonzero. The weighted homology groups may have different generators from the unweighted homology groups. 

The question may arise -- if the two homology theories are isomorphic, why consider weighted homology in this case? The key point is that the result in this section shows that two homology theories are isomorphic when the coefficient ring $R$ and the codomain of the weight function $w:K\to R$ are the same field $R=\mathbb{F}$. If they are different, for instance the weight function is $w:K\to\Z$ while the coefficient ring is $\Z/r\Z$, there could be difference in the two homology theories. This is discussed in Section \ref{sec:bockstein} on the Bockstein Spectral Sequence, in particular in Remark \ref{remark:basicmodular}.

Firstly, note that if $R=\mathbb{F}$ is a field, then the chain groups $C_n(K,w)$ are free $\mathbb{F}$-modules, or in other words vector spaces over $\mathbb{F}$. Then the kernel and image of the weighted boundary map, $\ker\partial_n$ and $\Ima \partial_n$ respectively, are also vector subspaces over $\mathbb{F}$.
\begin{lemma}
\label{kerneleq}
Let $(K,w)$ be a WSC with all weights of simplices nonzero. Let $\partial^w: C_n(K,w)\to C_{n-1}(K,w)$ and $\partial: C_n(K)\to C_{n-1}(K)$ denote the weighted boundary map and the usual unweighted boundary map respectively. Then $\ker\partial\cong\ker\partial^w$ as $\mathbb{F}$-vector spaces.
\end{lemma}
\begin{proof}
We consider the map $\psi: \ker\partial\to\ker\partial^w$, \[\psi\left(\sum_{i=1}^m a_i\sigma_i\right)=\sum_{i=1}^m\frac{a_i}{w(\sigma_i)}\sigma_i,\]
where $a_i\in\mathbb{F}$, and $\sigma_i$ are distinct basis elements of $C_n(K)$. 

The crucial part of the proof is to verify that $\sum_{i=1}^m\frac{a_i}{w(a_i)}\sigma_i\in\ker\partial^w$. Since \[\partial(\sum_{i=1}^ma_i\sigma_i)=\sum_{i=1}^m a_i\sum_{j=0}^n (-1)^jd_j(\sigma_i)=0,\] for each fixed basis element $d_k(\sigma_l)\in C_{n-1}(K)$, its coefficients must sum up to zero. That is,
\begin{equation*}
\sum_{\{i,j\mid d_j(\sigma_i)=d_k(\sigma_l)\}}a_i(-1)^j=0.
\end{equation*}

Note that
\begin{align}
\partial^w(\sum_{i=1}^m\frac{a_i}{w(\sigma_i)}\sigma_i)&=\sum_{i=1}^m\frac{a_i}{w(\sigma_i)}\partial^w(\sigma_i)\\
&=\sum_{i=1}^m\frac{a_i}{w(\sigma_i)}\left(\sum_{j=0}^n\frac{w(\sigma_i)}{w(d_j(\sigma_i))}(-1)^jd_j(\sigma_i)\right)\\
&=\sum_{i=1}^m a_i\sum_{j=0}^n\frac{1}{w(d_j(\sigma_i))}(-1)^jd_j(\sigma_i). \label{eq:cruciallastline}
\end{align}

Then for each basis element $d_k(\sigma_l)$ in the expression (\ref{eq:cruciallastline}), its coefficients sum up to 
\begin{align*}
&\sum_{\{i,j\mid d_j(\sigma_i)=d_k(\sigma_l)\}}\frac{a_i}{w(d_j(\sigma_i))}(-1)^j\\
&=\frac{1}{w(d_k(\sigma_l))}\sum_{\{i,j\mid d_j(\sigma_i)=d_k(\sigma_l)\}}a_i(-1)^j\\
&=0.
\end{align*}
Hence, $\partial^w(\sum_{i=1}^m \frac{a_i}{w(\sigma_i)}\sigma_i)=0$.

The map $\psi$ is clearly linear. Since $a_i\in\mathbb{F}$ and $w(\sigma_i)\in\mathbb{F}\setminus\{0\}$, hence $\frac{a_i}{w(\sigma_i)}\in\mathbb{F}$. If $\psi(\sum_{i=1}^m a_i\sigma_i)=\sum_{i=1}^m\frac{a_i}{w(\sigma_i)}\sigma_i=0$, then since the $\sigma_i$ are distinct basis elements of $C_n(K)$, thus $\frac{a_i}{w(\sigma_i)}=0$ for all $i$. Hence $a_i=0$ for all $i$, and $\sum_{i=1}^m a_i\sigma_i=0$. Hence $\psi$ is injective. For surjectivity, we observe that any $\sum_{i=1}^m b_i\sigma_i\in\ker\partial^w$ can be written in the form $\sum_{i=1}^m\frac{a_i}{w(\sigma_i)}$ by setting $a_i=b_iw(\sigma_i)$, where we can similarly check\footnote{The proof that $\sum_{i=1}^ma_i\sigma_i$ indeed lies in $\ker\partial$ is similar to the part where we verify that $\sum_{i=1}^m\frac{a_i}{w(a_i)}\sigma_i\in\ker\partial^w$.}  that $\sum_{i=1}^ma_i\sigma_i$ indeed lies in $\ker\partial$.

Therefore, we have shown that $\psi$ is a vector space isomorphism.
\end{proof}
\begin{lemma}
\label{imageeq}
Let $(K,w)$ be a WSC with all weights of simplices nonzero. Let $\partial^w: C_n(K,w)\to C_{n-1}(K,w)$ and $\partial: C_n(K)\to C_{n-1}(K)$ denote the weighted boundary map and the usual unweighted boundary map respectively. Then $\Ima\partial\cong\Ima\partial^w$ as $\mathbb{F}$-vector spaces.
\end{lemma}
\begin{proof}
Consider the map $\phi:\Ima\partial\to\Ima\partial^w$ by defining
\begin{equation*}
\phi\left(\sum_{i=0}^n (-1)^i d_i(\sigma)\right)=\sum_{i=0}^n\frac{w(\sigma)}{w(d_i(\sigma))}(-1)^id_i(\sigma),
\end{equation*}
where $\sigma\in C_n(K)$, and extending linearly over $\mathbb{F}$.

Surjectivity is clear. Let $\sigma_j\in C_n(K)$ for $j=1,\dots,m$. If
\begin{align*}
\phi(\partial(\sigma_1)+\dots+\partial(\sigma_j))&=\partial^w(\sigma_1)+\dots+\partial^w(\sigma_j)\\
&=\partial^w(\sigma_1+\dots+\sigma_j)\\
&=0,
\end{align*}
then $\sigma_1+\dots+\sigma_j\in\ker\partial^w=\ker\partial$ by Lemma \ref{kerneleq}. Thus, $\partial(\sigma_1)+\dots+\partial(\sigma_j)=\partial(\sigma_1+\dots+\sigma_j)=0$. Hence $\ker\phi=0$ and hence $\phi$ is injective. We have shown that $\phi$ is an isomorphism.
\end{proof}
\begin{theorem}
\label{thm:fieldisom}
Let $(K,w)$ be a finite (or finite-type\footnote{A WSC $(K,w)$ is said to be of finite-type if for each $n$, the number of $n$-simplices in $K$ is finite.}) WSC with all weights of simplices nonzero, and let $\mathbb{F}$ be a field. Then $H_n(K,w;\mathbb{F})\cong H_n(K;\mathbb{F})$.
\end{theorem}
\begin{proof}
Let $\partial^w_n$ and $\partial_n$ denote the $n$th weighted and unweighted boundary maps respectively. Considering the dimension over $\mathbb{F}$, we have
\begin{align*}
\dim H_n(K,w;\mathbb{F})&=\dim(\ker\partial^w_n/\Ima \partial^w_{n+1})\\
&=\dim(\ker\partial^w_n)-\dim(\Ima\partial^w_{n+1})\\
&=\dim(\ker\partial_n)-\dim(\Ima\partial_{n+1}) \tag{by Lemma \ref{kerneleq} and \ref{imageeq}}\\
&=\dim(\ker\partial_n/\Ima\partial_{n+1})\\
&=\dim H_n(K;\mathbb{F}).
\end{align*}

In the above computations, all dimensions are finite since $K$ is a finite (or finite-type) simplicial complex. Hence $H_n(K,w;\mathbb{F})\cong H_n(K;\mathbb{F})$ as $\mathbb{F}$-vector spaces.
\end{proof}

We also prove that if the weighted and unweighted homology are isomorphic for all WSCs $(K,w)$, then $R$ must be a field. Recall that the definition of weighted homology requires that $R$ is an integral domain with 1. (In particular, we do not consider $R=0$.)
\begin{theorem}
Let $R$ be an integral domain with 1. If $H_n(K,w;R)\cong H_n(K;R)$ for all WSCs $(K,w)$ and for all $n$, then $R$ is a field.
\end{theorem}
\begin{proof}
Suppose that $R$ is not a field. Let $a\in R$ be a nonzero non-unit so that $(a)\neq R$. 
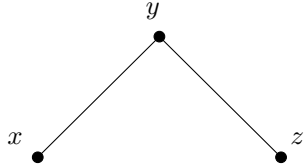
\begin{figure}[htbp]
\begin{center}
\begin{tikzpicture}[scale=0.8]
\draw (4.35,2.56) node[anchor=north west] {$x$};
\draw (6.62,4.72) node[anchor=north west] {$y$};
\draw (9.0,2.56) node[anchor=north west] {$z$};
\draw (5.,2.)-- (7.,4.);
\draw (9.,2.)-- (7.,4.);
\begin{scriptsize}
\draw [fill=black] (5.,2.) circle (2.5pt);
\draw [fill=black] (7.,4.) circle (2.5pt);
\draw [fill=black] (9.,2.) circle (2.5pt);
\end{scriptsize}
\end{tikzpicture}
\caption{Simplicial complex with 3 vertices $x$, $y$, $z$.}
\label{fig:torsioneg}
\end{center}
\end{figure}
Consider the WSC $(K,w)$ shown in Figure \ref{fig:torsioneg}, where $w(x)=1$, $w(y)=a$, $w(z)=1$, $w([x,y])=a$ and $w([y,z])=a$. 
Then
\begin{align*}
\partial_1([x,y])&=\frac{w([x,y])}{w(y)}y-\frac{w([x,y])}{w(x)}x\\
&=y-ax.
\end{align*}
Similarly, $\partial_1([y,z])=az-y$. Thus
\begin{align*}
H_0(K,w)&=\ker\partial_0/\Ima\partial_1\\
&\cong\langle x,y,z\mid y=ax,y=az\rangle\\
&\cong\langle x,z\mid ax=az\rangle\\
&\cong\langle x,x-z\mid a(x-z)=0\rangle\\
&\cong R\oplus R/(a)\not\cong R.
\end{align*}

On the other hand, $H_0(K)\cong R$ since $K$ is path-connected.
\end{proof}
\section{Bockstein Spectral Sequence and Weighted Persistent Homology}
\label{sec:bockstein}
Both spectral sequences and persistent homology are related to filtrations, hence it is natural to explore the relationship between them. In \cite{basu2017spectral}, Basu and Parida derived formulas which expresses the relationship between spectral sequences and persistent homology. In \cite{basu2017spectral}, all homology groups are taken with coefficients in a field. In \cite{romero2014defining}, Romero et al.\ study persistent $\mathbb{Z}$-homology using spectral sequences. We refer the reader to \cite{mccleary2001user,may2011more, chow2006you} for an overview of spectral sequences.

In this section we consider the Bockstein spectral sequence applied to weighted homology and weighted persistent homology. We will give a brief introduction to the Bockstein spectral sequence and refer the reader to \cite[ch.~10]{mccleary2001user}, \cite[ch.~24]{may2011more}, \cite[ch.~7]{neisendorfer2010algebraic} for more details. The motivation behind using the Bockstein spectral sequence is that in persistent homology algorithms \cite{Zomorodian2005,boissonnat2014computing}, most of the time the homology is computed with field coefficients, for instance $\mathbb{Z}/2\mathbb{Z}$. However, the integral homology groups contain more information than the homology groups with field coefficients. The Bockstein spectral sequence allows us to ``unravel'' the integral homology from the mod $p$ homology. Since the standard unweighted homology is a special case of weighted homology, the below results also hold for unweighted homology.
\subsection{Bockstein Homomorphism for Weighted Homology}
\label{ssec:bocksteinhom}
Recall the following results from \cite[p.~455]{mccleary2001user}, which we adapt to the context of weighted homology. Consider the short exact sequence of coefficient rings where $\red_r$ is reduction mod $r$:
\begin{equation*}
0\to\mathbb{Z}\xrightarrow{-\times r}\mathbb{Z}\xrightarrow{\red_r}\mathbb{Z}/r\mathbb{Z}\to 0.
\end{equation*}

The chain complex of a WSC $(K,w)$ with integer coefficients and weight function $w: K\to\mathbb{Z}$ is a complex, $C_*(K,w)$, of free abelian groups. It is clear that the maps $-\times r$ and $\red_r$ commute with the weighted boundary maps $\partial_n$. Hence, we obtain another short exact sequence of chain complexes (with integer coefficients)
\begin{equation}
\label{eq:shortexactchain}
0\to C_*(K,w)\xrightarrow{-\times r}C_*(K,w)\xrightarrow{\red_r}C_*(K,w)\otimes\mathbb{Z}/r\mathbb{Z}\to 0.
\end{equation}
Here, $\red_r$ is defined by $\red_r(c)=c\otimes 1$ for $c\in C_n(K,w)$.

\begin{lemma}
\label{lemma:longexact}
The short exact sequence of chain complexes (\ref{eq:shortexactchain}) induces a long exact sequence of homology groups:
\begin{equation*}
\begin{split}
&\dots\to H_{n+1}(K,w;\mathbb{Z}/r\mathbb{Z})\xrightarrow{\partial} H_n(K,w)\xrightarrow{-\times r}H_n (K,w)\xrightarrow{{\red_r}_*}H_n (K,w;\mathbb{Z}/r\mathbb{Z})\\&\xrightarrow{\partial} H_{n-1}(K,w)\to\dots
\end{split}
\end{equation*}
where
\begin{equation*}
H_n(K,w;\mathbb{Z}/r\mathbb{Z}):=\ker(\partial_n\otimes 1)/\Ima(\partial_{n+1}\otimes 1)
\end{equation*} and
\begin{equation*}
\partial_n\otimes 1: C_n(K,w)\otimes\mathbb{Z}/r\mathbb{Z}\to C_{n-1}(K,w)\otimes\mathbb{Z}/r\mathbb{Z}.
\end{equation*}
\end{lemma}
\begin{proof}
The proof is a standard double application of the Snake Lemma. See for instance \cite[pp.~121--122]{hilton2012course}.
\end{proof}
\begin{remark}
\label{remark:basicmodular}
Note that in general $H_n(K,w;\mathbb{Z}/r\mathbb{Z})$ with weight function $w: K\to\mathbb{Z}$ is different from $H_n(K,w';\mathbb{Z}/r\mathbb{Z})$ where $w': K\to\mathbb{Z}/r\mathbb{Z}$ is defined by $w'(\sigma)=w(\sigma) \pmod r$. This can be seen from the simple case of the 0-simplex $K=\{v_0\}$ with $w(v_0)=r$. Then $H_0(K,w;\mathbb{Z}/r\mathbb{Z})=\mathbb{Z}/r\mathbb{Z}$, while $H_0(K,w';\mathbb{Z}/r\mathbb{Z})=0$. This is because $C_0(K,w')=0$ due to the fact that $w'(v_0)=r\pmod r=0\pmod r$. However in the case that $w(\sigma)\neq 0\pmod r$ for all $\sigma\in K$, then $H_n(K,w;\mathbb{Z}/r\mathbb{Z})\cong H_n(K,w';\mathbb{Z}/r\mathbb{Z})$. This can be seen by observing that the boundary maps and chain groups in both cases are the same due to basic modular arithmetic.
\end{remark}
\begin{remark}
\label{remark:connecth}
We apply the construction of the connecting homomorphism in \cite[p.~99]{hilton2012course} to the situation in Lemma \ref{lemma:longexact}. Let $[c\otimes 1]\in H_{n+1}(K,w;\mathbb{Z}/r\mathbb{Z})$, where $c\in C_{n+1}(K,w)$. There exists $a\in\ker\partial_n$ such that $\partial_{n+1}(c)=ra$. Note that $a$ is unique by injectivity of the map $-\times r$. Then $\partial$ is defined by $\partial([c\otimes 1])=[a]\in H_n(K,w)$.
\end{remark}

If an element $u\in H_{n-1}(K,w)$ satisfies $ru=0$, i.e.\ $u\in\ker(-\times r)$, then by exactness $u\in\Ima(\partial)$. Hence there exists an element $\bar{u}\in H_n(K,w;\mathbb{Z}/r\mathbb{Z})$ such that $\partial(\bar{u})=u$. We write $\bar{u}=[c\otimes 1]\in H_n(K,w;\mathbb{Z}/r\mathbb{Z})$. Since $(\partial_n\otimes 1)(c\otimes 1)=0$, we conclude that $\partial_n(c)=rv$ for some $v\in\ker\partial_{n-1}$. By definition of the connecting homomorphism (see Remark \ref{remark:connecth}), $\partial$ takes $\bar{u}$ to $[v]\in H_{n-1}(K,w)$.
\begin{defn}[Bockstein homomorphism for weighted homology]
The \emph{Bockstein homomorphism} for weighted homology is defined by
\begin{equation*}
\begin{split}
\beta: H_n(K,w;\mathbb{Z}/r\mathbb{Z}) &\to H_{n-1}(K,w;\mathbb{Z}/r\mathbb{Z})\\
\bar{u}=[c\otimes 1] &\mapsto [v\otimes 1]=[\frac{1}{r}\partial_n c\otimes 1].
\end{split}
\end{equation*}
\end{defn}

The Bockstein spectral sequence is obtained from the long exact sequence in Lemma \ref{lemma:longexact} when we view it as an exact couple.
\subsection{The Bockstein Spectral Sequence for Weighted Homology}
Let $p$ be a prime number. Similar to the previous subsection \ref{ssec:bocksteinhom}, we can construct a long exact sequence associated to the short exact sequence of coefficients,
\begin{equation*}
0\to\mathbb{Z}\xrightarrow{-\times p}\mathbb{Z}\xrightarrow{\red_p}\mathbb{Z}/p\mathbb{Z}\to 0.
\end{equation*}

Notice that in the long exact sequence (Lemma \ref{lemma:longexact}), two out of every three terms is the same. Hence, we can interpret the long exact sequence as an exact couple \cite{browder1961torsion,mccleary2001user}:
\begin{equation*}
\begin{tikzcd}
H_*(K,w) \arrow[rr,"-\times p"] & & H_*(K,w) \arrow[dl,"{\red_p}_*"]\\
& H_*(K,w;\mathbb{Z}/p\mathbb{Z}) \arrow[ul,"\partial"] &
\end{tikzcd}
\end{equation*}

We define the $E^1$-term to be $E_n^1=H_n(K,w; \mathbb{Z}/p\mathbb{Z})$. The first differential is defined to be $d^1={\red_p}_*\circ\partial=\beta$, the Bockstein homomorphism. The resulting Bockstein spectral sequence is singly-graded.

\begin{theorem}
\label{thm:einfinity}
Let $(K,w)$ be a finite (or finite-type) WSC. Then there is a singly-graded spectral sequence $\{E_*^r, d^r\}$, with $E_n^1=H_n(K,w;\mathbb{Z}/p\mathbb{Z})$, $d^1=\beta$, the Bockstein homomorphism, and converging strongly to $(H_*(K,w)/\text{torsion})\otimes(\mathbb{Z}/p\mathbb{Z})$.
\end{theorem}
\begin{proof}
The proof is similar to \cite[Theorem~10.3]{mccleary2001user}.
\end{proof}

Following \cite[p.~460]{mccleary2001user}, we present an alternative and more direct presentation of the Bockstein homomorphism. Consider the short exact sequence of coefficients
\begin{equation*}
0\to\mathbb{Z}/p\mathbb{Z}\to\mathbb{Z}/p^2\mathbb{Z}\to\mathbb{Z}/p\mathbb{Z}\to 0.
\end{equation*}
The associated long exact sequence on weighted homology is given by
\begin{equation}
\label{eqn:bocksteinlong2}
\begin{split}
\dots&\to H_n(K,w;\mathbb{Z}/p\mathbb{Z})\xrightarrow{-\times p}H_n(K,w;\mathbb{Z}/p^2\mathbb{Z})\\
&\to H_n(K,w;\mathbb{Z}/p\mathbb{Z})\xrightarrow{\beta}H_{n-1}(K,w;\mathbb{Z}/p\mathbb{Z})\to\dots
\end{split}
\end{equation}
and has $d^1=\beta$, the connecting homomorphism. Similarly, when we consider the short exact sequence of coefficients
\begin{equation*}
0\to\mathbb{Z}/p^r\mathbb{Z}\to\mathbb{Z}/p^{2r}\mathbb{Z}\to\mathbb{Z}/p^r\mathbb{Z}\to 0,
\end{equation*}
we get the \emph{$r$-th order Bockstein operator} as the connecting homomorphism.

By an argument similar to \cite[Prop.~10.4]{mccleary2001user}, we obtain the following theorem.
\begin{theorem}
\label{thm:identifiedimage}
Let $\{E^r, d^r\}$ denote the Bockstein spectral sequence for weighted homology. $E_n^r$ is isomorphic to the subgroup of $H_n(K,w;\mathbb{Z}/p^r\mathbb{Z})$ given by the image of $H_n(K,w;\mathbb{Z}/p^r\mathbb{Z})\xrightarrow{-\times p^{r-1}}H_n(K,w;\mathbb{Z}/p^r\mathbb{Z})$ and $d^r: E_n^r\to E_{n-1}^r$ can be identified with the connecting homomorphism, the $r$-th order Bockstein homomorphism.
\qed
\end{theorem}

The $r$-th order Bockstein $\beta_r: H_n(K,w;\mathbb{Z}/p^r\mathbb{Z})\to H_{n-1}(K,w;\mathbb{Z}/p^r\mathbb{Z})$ maps $[c\otimes 1]$ to $[\frac{1}{p^r}\partial c\otimes 1]$. We also state a useful observation (Proposition \ref{prop:summand}), which is a generalization of a result in \cite[p.~481]{may2011more} to weighted homology. We note that since $\Z/p\Z$ is a field, $d^rE^r_{n+1}$ is a vector space isomorphic to a direct sum of $\Z/p\Z$ summands. In addition, we note that for a finite-type WSC $(K,w)$, $H_n(K,w)$ is a finitely generated abelian group which is canonically isomorphic to a direct sum of summands consisting of primary cyclic groups and infinite cyclic groups.

\begin{prop}
\label{prop:summand}
Let $(K,w)$ be a WSC of finite-type. There is a one-to-one correspondence between each summand $\mathbb{Z}/p\mathbb{Z}$ in the vector space $d^rE^r_{n+1}$, and each summand $\mathbb{Z}/p^r\mathbb{Z}$ in $H_{n}(K,w)$. In particular, there is a summand $\mathbb{Z}/p^r\mathbb{Z}$ in $H_*(K,w)$ if and only if the differential $d^r$ is nonzero.
\end{prop}
\begin{proof}
Each $i$th summand $\Z/p\Z$ in $d^rE_{n+1}^r$ corresponds to exactly one element $[c_i\otimes 1]\in H_{n+1}(K,w; \Z/p^r\Z)$ such that $\partial_{n+1}(c_i)=p^rv_i$ for some $v_i\in\ker\partial_n$. Furthermore the $v_i$ are linearly independent over $\mathbb{Z}$, in particular $v_i\neq v_j$ if $i\neq j$. Hence each $v_i+\Ima\partial_{n+1}$ generates a summand $\Z/p^r\Z$ in $H_n(K,w)$.

Conversely, each $i$th summand $\Z/p^r\Z$ in $H_n(K,w)$ is generated by $v_i'+\Ima\partial_{n+1}$, where $p^rv_i'=\partial_{n+1}(c_i')$ for some $c_i'\in C_{n+1}(K,w)$. Each $v_i'$ is distinct and the $v_i'$ are linearly independent over $\Z$. Hence the $c_i'$ are also distinct and linearly independent over $\Z$. Then for each $[c_i'\otimes 1]\in H_{n+1}(K,w;\Z/p^r\Z)$, we have $d^r[c_i'\otimes 1]=[\frac{1}{p^r}\partial_{n+1}c_i'\otimes 1]\neq 0$ which generates one summand $\Z/p\Z$ in $d^r E_{n+1}^r$.

In particular, there is a summand $\Z/p^r\Z$ in $H_*(K,w)$ iff there is a summand $\Z/p\Z$ in $d^r E_{n+1}^r$ iff $d^r$ is nonzero.
\end{proof}
\subsection{Applications}
For a finite (or finite-type) WSC $(K,w)$, a complete knowledge of the Bockstein spectral sequences for all primes $p$ allows us to recover completely the integral weighted homology $H_*(K,w)$. From Theorem \ref{thm:einfinity}, the $E^\infty$ term tells us the torsion-free part of $H_*(K,w)$. Moreover, by Proposition \ref{prop:summand}, the rank of the differential $d^r$ tells us the number of summands of $\mathbb{Z}/p^r\mathbb{Z}$ in the integral weighted homology. 

Hence, in the event that the Bockstein spectral sequence is known or has already been computed, we can skip the calculation of the integral weighted homology, and instead derive it from the Bockstein spectral sequence. We illustrate the above idea with an example.

\begin{eg}
Consider the WSC $(K,w)$ shown in Figure \ref{fig:bocksteineg}.
\begin{figure}[htbp]
\begin{center}
\begin{tikzpicture}[scale=0.8]
\draw (4.22,2.56) node[anchor=north west] {$v_0$};
\draw (6.62,4.72) node[anchor=north west] {$v_2$};
\draw (9.0,2.56) node[anchor=north west] {$v_1$};
\draw (5.,2.)-- (7.,4.);
\draw (9.,2.)-- (7.,4.);
\draw (5.,2.)-- (9.,2.);
\begin{scriptsize}
\draw [fill=black] (5.,2.) circle (2.5pt);
\draw [fill=black] (7.,4.) circle (2.5pt);
\draw [fill=black] (9.,2.) circle (2.5pt);
\end{scriptsize}
\end{tikzpicture}
\caption{WSC $(K,w)$ with the following weights: $w(v_0)=w(v_1)=w(v_2)=1$, $w([v_0,v_1])=w([v_1,v_2])=w([v_0,v_2])=4$.}
\label{fig:bocksteineg}
\end{center}
\end{figure}
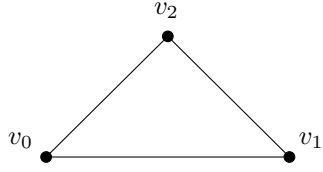

We first compute the Bockstein spectral sequence for $p=2$. We get the following results, where the notation $\mathbb{Z}/p$ is short for $\mathbb{Z}/p\mathbb{Z}$.

\begin{align*}
E_n^1&=\begin{cases}H_0(K,w;\mathbb{Z}/2)\cong\mathbb{Z}/2\oplus\mathbb{Z}/2\oplus\mathbb{Z}/2,&\text{for}\ n=0\\
H_1(K,w;\mathbb{Z}/2)\cong\Z/2,&\text{for}\ n=1.
\end{cases}\\
d^1&=0\\
E_n^2&\cong\begin{cases}\Z/2\oplus\Z/2\oplus\Z/2,&\text{for}\ n=0\\
\Z/2\oplus\Z/2\oplus\Z/2,&\text{for}\ n=1.
\end{cases}\\
d^2E_1^2&\cong\Z/2\oplus\Z/2\\
E_n^3&\cong\begin{cases}\Z/2,&\text{for}\ n=0\\
\Z/2,&\text{for}\ n=1.
\end{cases}\\
d^r&=0\quad\text{for}\ r\geq 3\\
E_n^\infty&\cong\begin{cases}\Z/2,&\text{for}\ n=0\\
\Z/2,&\text{for}\ n=1.
\end{cases}
\end{align*}

The interpretation of the above results is as follows. Firstly, since $E_0^\infty=E_1^\infty=\mathbb{Z}/2\mathbb{Z}$, we conclude using Theorem \ref{thm:einfinity} that 
\begin{equation*}
H_0(K,w)/\text{torsion}=H_1(K,w)/\text{torsion}=\mathbb{Z}.
\end{equation*}
From the differentials, we conclude that $H_0(K,w)$ has exactly 2 summands of $\mathbb{Z}/4\mathbb{Z}$, and $H_*(K,w)$ has no summands of the form $\mathbb{Z}/2^r\mathbb{Z}$ for $r\neq 2$.

For $p\neq 2$, we can compute that $d^r=0$ for all $r$. This is due to the weighted boundary map $\partial_1$ which produces output that are not divisible by $p$ for $p\neq 2$. For instance, $\partial_1([v_0,v_1])=4v_1-4v_0$. Hence, elements in $E_n^r$ are cycles and thus $d^r(E_n^r)=0$. We conclude that $H_*(K,w)$ has no summands of the form $\mathbb{Z}/p^r\mathbb{Z}$ for $p\neq 2$.

Combining the above information, we recover the integral weighted homology:
\begin{equation*}
H_n(K,w)=\begin{cases}
\Z\oplus\Z/4\oplus\Z/4, &\text{for}\ n=0\\
\Z, &\text{for}\ n=1.
\end{cases}
\end{equation*}
\end{eg}
\subsection{Application to Weighted Persistent Homology}
Let $(\mathcal{K},w)=\{(K^i,w)\}_{i\geq 0}$ be a weighted filtered complex. Let $H_k^i(\mathcal{K},w)$ (or $H_k^i$ for short) denote the $k$-th weighted homology group of the $i$-th complex $K^i$. It is known that the homomorphism 
\begin{equation*}
\begin{split}
\eta_k^{i,q}: H_k^i&\to H_k^{i+q}\\
\alpha+B_k^i&\mapsto\alpha+B_k^{i+q}
\end{split}
\end{equation*}
is well-defined, and $\Ima\eta_k^{i,q}\cong H_k^{i,q}$ (cf.\ \cite[p.~2680]{ren2018weighted}).
\begin{defn}
Let $H_k^i(\mathcal{K},w;\mathbb{Z}/p\mathbb{Z})$ denote the $k$-th weighted homology group of the $i$-th complex $K^i$, with coefficients in $\Z/p\Z$.
We define the map
\begin{equation*}
\begin{split}
\theta_k^{i,q}: H_k^i(\mathcal{K},w;\mathbb{Z}/p\mathbb{Z})&\to H_k^{i+q}(\mathcal{K},w;\mathbb{Z}/p\mathbb{Z})\\
c\otimes 1+\Ima(\partial_{k+1}^i\otimes 1)&\mapsto c\otimes 1+\Ima(\partial_{k+1}^{i+q}\otimes 1)
\end{split}
\end{equation*}
where $\partial_k^i\otimes 1$ is the map
\begin{equation*}
\partial_k^i\otimes 1: C_k(K^i,w)\otimes\mathbb{Z}/p\mathbb{Z}\to C_{k-1}(K^i,w)\otimes\mathbb{Z}/p\mathbb{Z}.
\end{equation*}
\end{defn}

The map $\theta_k^{i,q}$ is well-defined, since if $c_1\otimes 1+\Ima(\partial_{k+1}^i\otimes 1)=c_2\otimes 1+\Ima(\partial_{k+1}^{i}\otimes 1)$, then $(c_1-c_2)\otimes 1\in\Ima(\partial_{k+1}^i\otimes 1)\subseteq\Ima(\partial_{k+1}^{i+q}\otimes 1)$. Similarly, we also have
\begin{equation*}
\Ima\theta_k^{i,q}\cong H_k^{i,q}(\mathcal{K},w;\Z/p\Z).
\end{equation*}
\begin{prop}
\label{prop:etatheta}
The following is a commutative diagram with exact rows for all $i,q\geq 0$:
\begin{equation*}
\begin{tikzcd}[column sep=scriptsize]
\dots\arrow[r] &H_{k+1}^i(\mathcal{K},w;\Z/p)\arrow[r,"\partial"]\arrow[d,"\theta_{k+1}^{i,q}"] &H_k^i(\mathcal{K},w) \arrow[r,"-\times p"]\arrow[d,"\eta_k^{i,q}"] &H_k^i(\mathcal{K},w) \arrow[r,"{\red_p}_*"]\arrow[d,"\eta_k^{i,q}"] &H_k^i(\mathcal{K},w;\Z/p) \arrow[r]\arrow[d,"\theta_{k}^{i,q}"]  &\dots\\
\dots\arrow[r] &H_{k+1}^{i+q}(\mathcal{K},w;\Z/p)\arrow[r,"\partial"] &H_k^{i+q}(\mathcal{K},w) \arrow[r,"-\times p"] &H_k^{i+q}(\mathcal{K},w) \arrow[r,"{\red_p}_*"] &H_k^{i+q}(\mathcal{K},w;\Z/p) \arrow[r] &\dots
\end{tikzcd}
\end{equation*}
\end{prop}
\begin{proof}
The exactness of the rows is due to Lemma \ref{lemma:longexact}, obtained in the process of constructing the Bockstein spectral sequence.

We check that each square commutes.
\begin{align*}
\partial\theta_{k+1}^{i,q}(c\otimes 1+\Ima(\partial_{k+2}^i\otimes 1))&=\eta_k^{i,q}\partial(c\otimes 1+\Ima(\partial_{k+2}^i\otimes 1))\\
&=\frac{1}{p}\partial_{n+1}^{i+q}c+\Ima(\partial_{k+1}^{i+q})\\
(-\times p)\eta_k^{i,q}(\alpha+\Ima(\partial_{k+1}^i))&=\eta_k^{i,q}(-\times p)(\alpha+\Ima(\partial_{k+1}^i))\\
&=p\alpha+\Ima(\partial_{k+1}^{i+q})\\
{\red_p}_*\eta_k^{i,q}(\alpha+\Ima(\partial_{k+1}^i))&=\theta_k^{i,q}{\red_p}_*(\alpha+\Ima(\partial_{k+1}^i))\\
&=\alpha\otimes 1+\Ima(\partial_{k+1}^{i+q}\otimes 1).
\end{align*}
\end{proof}
\begin{prop}
\label{prop:commdiagram2}
The following is a commutative diagram with exact rows for all $i,q\geq 0$:
\begin{equation*}
\begin{tikzcd}[column sep=tiny]
\dots\arrow[r] &H_{k}^i(\mathcal{K},w;\Z/p)\arrow[r,"-\times p"]\arrow[d,"\theta_{k}^{i,q}"] &H_k^i(\mathcal{K},w;\Z/p^2) \arrow[r]\arrow[d,"\epsilon_k^{i,q}"] &H_k^i(\mathcal{K},w;\Z/p) \arrow[r,"\beta"]\arrow[d,"\theta_k^{i,q}"] &H_{k-1}^i(\mathcal{K},w;\Z/p) \arrow[r]\arrow[d,"\theta_{k-1}^{i,q}"]  &\dots\\
\dots\arrow[r] &H_{k}^{i+q}(\mathcal{K},w;\Z/p)\arrow[r,"-\times p"] &H_k^{i+q}(\mathcal{K},w;\Z/p^2) \arrow[r] &H_k^{i+q}(\mathcal{K},w;\Z/p) \arrow[r,"\beta"] &H_{k-1}^{i+q}(\mathcal{K},w;\Z/p) \arrow[r] &\dots
\end{tikzcd}
\end{equation*}
where $\epsilon_k^{i,q}$ is defined similarly to $\theta_k^{i,q}$. That is, 
\begin{equation*}
\begin{split}
\epsilon_k^{i,q}: H_k^i(\mathcal{K},w;\Z/p^2)&\to H_k^{i+q}(\mathcal{K},w;\Z/p^2)\\
c\otimes 1+\Ima(\partial_{k+1}^i\otimes 1)&\mapsto c\otimes 1+\Ima(\partial_{k+1}^{i+q}\otimes 1).
\end{split}
\end{equation*}
\begin{proof}
The exactness of the rows is due to the long exact sequence (\ref{eqn:bocksteinlong2}). The commutativity of each square can be verified similarly to Proposition \ref{prop:etatheta}.
\end{proof}
\end{prop}
\begin{remark}
Similarly, we have that $\Ima\epsilon_{k}^{i,q}\cong H_k^{i,q}(\mathcal{K},w;\Z/p^2\Z)$.
\end{remark}

A direct application of the Four Lemma \cite[p.~364]{maclane2012homology} to the commutative diagram in Proposition \ref{prop:commdiagram2} gives us the following lemma.
\begin{lemma}
\label{thm:fourlemmaresult}
\begin{enumerate}
\item If $\theta_k^{i,q}$ is surjective, and $\epsilon_k^{i,q}$ and $\theta_{k-1}^{i,q}$ are injective, then $\theta_k^{i,q}$ is injective.
\item If $\theta_{k-1}^{i,q}$ is injective, and $\theta_k^{i,q}$ is surjective, then $\epsilon_k^{i,q}$ is surjective.
\end{enumerate}
\end{lemma}

In particular, the second statement of Lemma \ref{thm:fourlemmaresult} has applications to calculate the mod $p^2$ weighted persistent homology given some information about the mod $p$ persistent homology. We describe it in the following theorem.
\begin{theorem}
\label{cor:ptop2}
Let $k,i,q\geq 0$. If both statements
\begin{equation}
\label{eqn:1sthomstatement}
H_{k-1}^{i,q}(\mathcal{K},w;\Z/p)\cong H_{k-1}^i(\mathcal{K},w;\Z/p),
\end{equation}
\begin{equation}
\label{eqn:2ndhomstatement}
H_k^{i,q}(\mathcal{K},w;\Z/p)\cong H_{k}^{i+q}(\mathcal{K},w;\Z/p),
\end{equation}
are true, then
\begin{equation*}
H_k^{i,q}(\mathcal{K},w;\Z/p^2)\cong H_k^{i+q}(\mathcal{K},w;\Z/p^2).
\end{equation*}

(When $k=0$, by convention we let $H_{k-1}^{i,q}$ and $H_{k-1}^i$ to be 0.)
\end{theorem}
\begin{proof}
By the first isomorphism theorem,
\begin{equation*}
H_{k-1}^i(\mathcal{K},w;\Z/p)/\ker\theta_{k-1}^{i,q}\cong\Ima\theta_{k-1}^{i,q}\cong H_{k-1}^{i,q}(\mathcal{K},w;\Z/p).
\end{equation*}
Hence, if (\ref{eqn:1sthomstatement}) holds, then $\theta_{k-1}^{i,q}$ is injective. Similarly, if (\ref{eqn:2ndhomstatement}) holds, then $\theta_k^{i,q}$ is surjective. By the second statement of Lemma \ref{thm:fourlemmaresult}, $\epsilon_{k}^{i,q}$ is surjective, i.e.
\begin{equation*}
H_k^{i,q}(\mathcal{K},w;\Z/p^2)\cong H_k^{i,q}(\mathcal{K},w;\Z/p^2).
\end{equation*}
\end{proof}
\begin{remark}
By considering the long exact sequence of homology associated to the short exact sequence of coefficients $0\to\Z/p^r\to\Z/p^{2r}\to\Z/p^r\to 0$ we can generalize Theorem \ref{cor:ptop2} to higher powers of $p$. That is, the conclusion of Theorem \ref{cor:ptop2} still holds if we replace $\Z/p$ by $\Z/p^r$ and $\Z/p^2$ by $\Z/p^{2r}$.
\end{remark}

The condition $H_{k-1}^{i,q}(\mathcal{K},w;\Z/p)\cong H_{k-1}^i(\mathcal{K},w;\Z/p)$ (\ref{eqn:1sthomstatement}) is necessary for Theorem \ref{cor:ptop2}. Without it, Theorem \ref{cor:ptop2} may not be true, as the following counterexample shows.

\begin{eg}
Consider the filtration of WSCs as shown in Figure \ref{fig:countereg}.
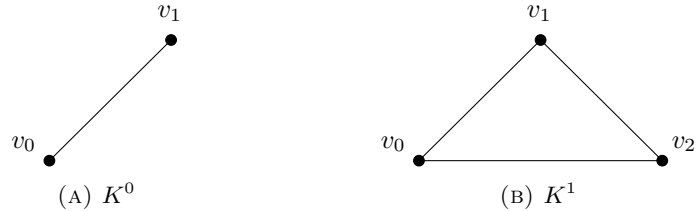
\begin{figure}[htbp]
\begin{subfigure}[]{.45\linewidth}
\centering
\begin{tikzpicture}[scale=0.8]
\draw (4.22,2.56) node[anchor=north west] {$v_0$};
\draw (6.62,4.72) node[anchor=north west] {$v_1$};
\draw (5.,2.)-- (7.,4.);
\begin{scriptsize}
\draw [fill=black] (5.,2.) circle (2.5pt);
\draw [fill=black] (7.,4.) circle (2.5pt);
\end{scriptsize}
\end{tikzpicture}
\caption{$K^0$}
\end{subfigure}
\begin{subfigure}[]{.45\linewidth}
\centering
\begin{tikzpicture}[scale=0.8]
\draw (4.22,2.56) node[anchor=north west] {$v_0$};
\draw (6.62,4.72) node[anchor=north west] {$v_1$};
\draw (9.0,2.56) node[anchor=north west] {$v_2$};
\draw (5.,2.)-- (7.,4.);
\draw (9.,2.)-- (7.,4.);
\draw (5.,2.)-- (9.,2.);
\begin{scriptsize}
\draw [fill=black] (5.,2.) circle (2.5pt);
\draw [fill=black] (7.,4.) circle (2.5pt);
\draw [fill=black] (9.,2.) circle (2.5pt);
\end{scriptsize}
\end{tikzpicture}
\caption{$K^1$}
\end{subfigure}
\caption{The filtration of WSCs with the following weights: $w([v_0,v_1])=2$, $w(\sigma)=1$ for all other simplices $\sigma\neq [v_0,v_1]$.}
\label{fig:countereg}
\end{figure}

We have that
\begin{align*}
H_0^0(\mathcal{K},w;\Z/2)&\cong\Z/2\oplus\Z/2\\
H_0^{0,1}(\mathcal{K},w;\Z/2)&\cong\Z/2\\
H_1^1(\mathcal{K},w;\Z/2)&\cong\Z/2\\
H_1^{0,1}(\mathcal{K},w;\Z/2)&\cong\Z/2.
\end{align*}

That is, the first condition of Theorem \ref{cor:ptop2} is \emph{not} satisfied, but the second condition is satisfied. The conclusion of Theorem \ref{cor:ptop2} does not hold: 
\begin{align*}
H_1^1(\mathcal{K},w;\Z/4)&\cong\Z/4\\
H_1^{0,1}(\mathcal{K},w;\Z/4)&\cong 0.
\end{align*}
\end{eg}

The condition (\ref{eqn:2ndhomstatement}) in Theorem \ref{cor:ptop2} is also necessary. If (\ref{eqn:2ndhomstatement}) is not satisfied, we can construct a simple counterexample.
\begin{eg}
Consider $K^0=\{v_0\}$, $K^1=\{v_0,v_1\}$, with $w(v_0)=w(v_1)=1$. Then, we have:
\begin{align*}
H_0^{0,1}(\mathcal{K},w;\Z/2)\cong\Z/2\\
H_0^1(\mathcal{K},w;\Z/2)\cong\Z/2\oplus\Z/2\\
H_0^{0,1}(\mathcal{K},w;\Z/4)\cong\Z/4\\
H_0^1(\mathcal{K},w;\Z/4)\cong\Z/4\oplus\Z/4.
\end{align*}
\end{eg}
\section{Generalized Bockstein Spectral Sequence for Weighted Homology}
\label{sec:genbockstein}
In \cite[pp.~465--490]{lubkin1980cohomology}, a generalized Bockstein spectral sequence of the cochain complex $C^*$ with respect to a fixed element $t$ in the center of a ring $A$ was studied. In this section, we study and develop a generalized Bockstein spectral sequence in the context of weighted homology.

Let $R$ be an integral domain with 1. Let $(K,w)$ be a WSC with weight function $w: K\to R$. Let $a\in R\setminus\{0\}$ be a fixed element of $R$.

Consider the short exact sequence of coefficient rings
\begin{equation*}
0\to R\xrightarrow{\mu_a}R\xrightarrow{\rho_a}R\otimes R/aR\to 0
\end{equation*}
where $\mu_a(r)=ar$ and $\rho_a(r)=r\otimes 1$.

The chain complex $C_*(K,w)$ (over $R$) is a free $R$-module, and we obtain a short exact sequence of chain complexes
\begin{equation*}
0\to C_*(K,w)\xrightarrow{\mu_a} C_*(K,w)\xrightarrow{\rho_a}C_*(K,w)\otimes R/aR\to 0.
\end{equation*}

From that, we get a long exact sequence of homology groups, 
\begin{equation*}
\cdots\to H_n(K,w; R)\xrightarrow{{\mu_a}_*}H_n(K,w; R)\xrightarrow{{\rho_a}_*}H_n(K,w; R/aR)\xrightarrow{\partial}H_{n-1}(K,w; R)\to\cdots.
\end{equation*}

\begin{defn}
The \emph{generalized Bockstein homomorphism} is defined by
\begin{equation*}
\begin{split}
\beta: H_n(K,w; R/aR)&\to H_{n-1}(K,w; R/aR)\\
[c\otimes 1]&\mapsto [\frac{1}{a}\partial c\otimes 1].
\end{split}
\end{equation*}
\end{defn}
\begin{remark}
Note that if $R=\mathbb{Z}$ and $a\in\mathbb{Z}$, then $\beta$ is the usual Bockstein homomorphism \cite[p.~456]{mccleary2001user}.
\end{remark}
\subsection{The Generalized Bockstein Spectral Sequence for Weighted Homology}
Let $p$ be a prime element in the integral domain $R$. To set up the \emph{generalized Bockstein spectral sequence}, we view the long exact sequence as an exact couple:
\begin{equation*}
\begin{tikzcd}
H_*(K,w;R) \arrow[rr,"{\mu_p}_*"] & & H_*(K,w; R) \arrow[dl,"{\rho_p}_*"]\\
& H_*(K,w; R/pR) \arrow[ul,"\partial"] &
\end{tikzcd}
\end{equation*}

We define the $E^1$-term to be $E_n^1=H_n(K,w; R/pR)$, and the first differential to be $d^1={\rho_p}_*\circ\partial=\beta$, the generalized Bockstein homomorphism.
\begin{theorem}
\label{thm:genconv}
Let $R$ be a PID. Let $(K,w)$ be a finite (or finite-type) WSC, with weight function $w: K\to R$. Then there is a singly-graded spectral sequence $\{E_*^r, d^r\}$, with $E_n^1=H_n(K,w; R/pR)$, $d^1=\beta$, the generalized Bockstein homomorphism, and converging strongly to $(H_*(K,w;R)/\text{torsion})\otimes (R/pR)$.
\end{theorem}
\begin{proof}
The proof is similar to \cite[Theorem 10.3]{mccleary2001user}. We need $R$ to be a PID in order to use the structure theorem for finitely generated modules over a PID.
\end{proof}

Consider the short exact sequence of coefficients
\begin{equation*}
0\to R/p^rR\to R/p^{2r}R\to R/p^rR\to 0,
\end{equation*}
where we have used the isomorphism $R/p^rR\cong p^rR/p^{2r}R$. We obtain the \emph{$r$-th order generalized Bockstein operator} as connecting homomorphism of the associated long exact sequence on homology. 

The generalized version of Theorem \ref{thm:identifiedimage} also holds:
\begin{theorem}
Let $R$ be an integral domain. Let $\{E^r, d^r\}$ denote the generalized Bockstein spectral sequence for weighted homology. $E_n^r$ is isomorphic to the subgroup of $H_n(K,w; R/p^rR)$ given by the image of $H_n(K,w; R/p^r R)\xrightarrow{\mu_{p^{r-1}}}H_n(K,w; R/p^rR)$ and $d^r: E_n^r\to E_{n-1}^r$ can be identified with the connecting homomorphism, the $r$-th order generalized Bockstein homomorphism.
\qed
\end{theorem}

We also have the following generalization of Proposition \ref{prop:summand}.

\begin{prop}
\label{thm:gencorr}
Let $(K,w)$ be a WSC of finite-type. There is a one-to-one correspondence between each summand $R/pR$ in $d^rE_{n+1}^r$, and each summand $R/p^rR$ in $H_n(K,w; R)$. In particular, there is a summand $R/p^rR$ in $H_*(K,w; R)$ if and only if the differential $d^r$ is nonzero.
\qed
\end{prop}


We show an example of the generalized Bockstein spectral sequence over the polynomial ring $\R[x]$.

\begin{eg}
Let $R$ be the polynomial ring $\R[x]$. Consider the WSC $(K,w)$, $w: K\to\mathbb{R}[x]$, as  shown in Figure \ref{fig:genbocksteineg}.
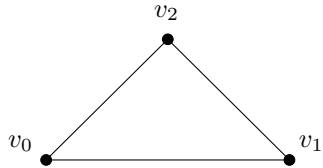
\begin{figure}[htbp]
\begin{center}
\begin{tikzpicture}[scale=0.8]
\draw (4.22,2.56) node[anchor=north west] {$v_0$};
\draw (6.62,4.72) node[anchor=north west] {$v_2$};
\draw (9.0,2.56) node[anchor=north west] {$v_1$};
\draw (5.,2.)-- (7.,4.);
\draw (9.,2.)-- (7.,4.);
\draw (5.,2.)-- (9.,2.);
\begin{scriptsize}
\draw [fill=black] (5.,2.) circle (2.5pt);
\draw [fill=black] (7.,4.) circle (2.5pt);
\draw [fill=black] (9.,2.) circle (2.5pt);
\end{scriptsize}
\end{tikzpicture}
\caption{WSC $(K,w)$ with the following weights: $w(v_0)=w(v_1)=w(v_2)=1$, $w([v_0,v_1])=w([v_1,v_2])=w([v_0,v_2])=x^2$.}
\label{fig:genbocksteineg}
\end{center}
\end{figure}

We first compute the generalized Bockstein spectral sequence for $p=x$. We obtain the following results.
\begin{align*}
E_n^1&=\begin{cases}
H_0(K,w; \R[x]/(x))\cong(\R[x]/(x))^3, &\text{for $n=0$}\\
H_1(K,w; \R[x]/(x))\cong\R[x]/(x), &\text{for $n=1$.}
\end{cases}\\
d^1&=0\\
E_n^2&\cong\begin{cases}
(\R[x]/(x))^3, &\text{for $n=0$}\\
(\R[x]/(x))^3, &\text{for $n=1.$}
\end{cases}\\
d^2E_1^2&\cong (\R[x]/(x))^2\\
E_n^3&\cong\begin{cases}
\R[x]/(x), &\text{for $n=0$}\\
\R[x]/(x), &\text{for $n=1$.}
\end{cases}\\
d^r&=0\quad\text{for $r\geq 3$}\\
E_n^\infty&\cong\begin{cases}
\R[x]/(x), &\text{for $n=0$}\\
\R[x]/(x), &\text{for $n=1$}.
\end{cases}
\end{align*}

By Theorem \ref{thm:genconv} and Proposition \ref{thm:gencorr}, we can recover the $\R[x]$-homology:
\begin{equation*}
H_n(K,w; \R[x])\cong\begin{cases}
\R[x]\oplus\R[x]/(x^2)\oplus\R[x]/(x^2), &\text{for $n=0$}\\
\R[x], &\text{for $n=1$}.
\end{cases}
\end{equation*}
\end{eg}
\bibliographystyle{amsplain}
\bibliography{jabref9}

\end{document}